\documentclass[12pt,a4paper,oneside]{amsart}
\usepackage{fullpage}
\usepackage{amssymb,amsfonts}
\usepackage{epsfig}
\usepackage{graphicx}

\usepackage{graphics}
\newtheorem{theorem}{Theorem}[section]
\newtheorem{lemma}[theorem]{Lemma}

\theoremstyle{definition}
\newtheorem{definition}[theorem]{Definition}
\newtheorem{remark}[theorem]{Remark}

\numberwithin{equation}{section}

\DeclareMathOperator{\diam}{diam} 

\newcommand{\be}{\begin{equation}}
\newcommand{\ee}{\end{equation}}

\newcommand{\dist}{{\operatorname{dist}}}

\DeclareMathOperator{\rad}{rad}

\makeindex

\def\Xint#1{\mathchoice 
 {\XXint\displaystyle\textstyle{#1}}%
{\XXint\textstyle\scriptstyle{#1}}%
{\XXint\scriptstyle\scriptscriptstyle{#1}}%
 {\XXint\scriptscriptstyle\scriptscriptstyle{#1}}%
 \!\int}
\def\XXint#1#2#3{{\setbox0=\hbox{$#1{#2#3}{\int}$}
 \vcenter{\hbox{$#2#3$}}\kern-.5\wd0}}

 \def\dashint{\Xint-}

\begin{document}
\title{Lebesgue points via the Poincar\'e inequality}
\author{Nijjwal Karak}
\address{Department of Mathematics and Statistics, University of Jyv\"askyl\"a, P.O. Box 35, FI-40014, Jyv\"askyl\"a, Finland}
\email{nijjwal.n.karak@jyu.fi}
\author{Pekka Koskela}
\address{Department of Mathematics and Statistics, University of Jyv\"askyl\"a, P.O. Box 35, FI-40014, Jyv\"askyl\"a, Finland}
\email{pekka.j.koskela@jyu.fi}
\thanks{The authors were partially supported by the Academy of Finland grant number 131477}
\begin{abstract}
In this article, we show that in a $Q$-doubling space $(X,d,\mu),$ $Q>1,$ which satisfies a chain condition, if we have a $Q$-Poincar\'e inequality for a pair of functions $(u,g)$ where $g\in L^Q(X),$ then $u$ has Lebesgue points $H^h$-a.e. for $h(t)=\log^{1-Q-\epsilon}(1/t).$ We also discuss how the existence of Lebesgue points follows for $u\in W^{1,Q}(X)$ where $(X,d,\mu)$ is a complete $Q$-doubling space supporting a $Q$-Poincar\'e inequality for $Q>1.$ 
\end{abstract}
\maketitle
\section{Introduction}
The usual argument for obtaining the existence of Lebesgue points outside
a small set for a Sobolev function $u\in W$ goes as follows. First of all, 
Lebesgue points exist except for a set of $W$-capacity zero \cite{HKM06},
\cite{MZ97}; this is proven by approximating $u$ by continuous functions. Secondly, each set of positive Hausdorff $h$-measure, for a suitable $h,$
is of positive $W$-capacity, see Theorem 7.1 in \cite{KM72} or Theorem 5.1.13 in \cite{AH96}.\\
\indent For the usual euclidean Sobolev space $W^{1,n}(\mathbb{R}^n),$ this argument shows that,
given $\epsilon>0,$ a function $u\in W^{1,n}(\mathbb{R}^n)$ satisfies
\begin{equation} \label{differentioituvuus}
u(x)=\lim_{r\to 0}\frac{1}{|B(x,r)|}\int_{B(x,r)}u(y)\, dy
\end{equation}
outside a set $E_\epsilon$ with $H^{h}(E_\epsilon)=0,$
where $h(t)=\log^{1-n-\epsilon}(1/t).$ In fact, any non-decreasing non-negative gauge function $h$ that satisfies
\begin{equation}\label{gauge}
\int_{0}^{1}h(t)^{1/(n-1)}\, \frac{dt}{t}<\infty
\end{equation}
can be used. To be precise, a function $u\in W^{1,n}(\mathbb{R}^n)$ is a priori only defined almost everywhere with respect to the $n$-dimensional measure. The meaning of \eqref{differentioituvuus} is that the limit of integral averages of $u$ exists $H^h$-a.e. and after replacing $u$ with this
limit, we obtain a representative of $u$ for which \eqref{differentioituvuus} holds outside $E_\epsilon.$\\
\indent The above argument is very general. Let us consider a doubling metric
space $(X,d,\mu).$ Then a simple iteration argument shows that there
is an exponent $Q>0$ and a constant $C\geq 1$ so that
\begin{equation}\label{1}
\left(\frac{s}{r}\right)^Q\leq C\frac{\mu(B(x,s))}{\mu(B(a,r))}
\end{equation}
holds whenever $a\in X$, $x\in B(a,r)$ and $0<s\leq r.$ We say that $(X,d,\mu)$ is $Q$-\textit{doubling} if $(X,d,\mu)$ is a doubling metric measure space and \eqref{1} holds with the given $Q.$ Towards defining our Sobolev space, we recall that a Borel-measurable function $g\ge 0$ is an upper gradient of a measurable function $u$ provided
\begin{equation}
\vert u(\gamma(a))-u(\gamma(b))\vert\le\int_{\gamma}g\, ds
\end{equation}
for every rectifiable curve $\gamma : [a,b]\rightarrow X$ \cite{HK98}, \cite{KM98}. We define $W^{1,p}(X),$ $1\le p<\infty,$ to be the collection of all $u\in L^p(X)$ that have an upper gradient that also belongs to $L^p(X),$ see \cite{Sha00}. In order to obtain lower bounds for the capacity associated to  $W^{1,p}(X),$ it suffices to assume a suitable Poincar\'e inequality. We say that $(X,d,\mu)$ supports a $p$-Poincar\'e inequality if there exist constants $C$ and $\lambda$ such that
\begin{equation}\label{PI}
\dashint_B\vert u-u_B\vert\, d\mu\le C\diam(B)\left(\dashint_{\lambda B}g^p\, d\mu\right)^{1/p}
\end{equation}
for every open ball $B$ in $X$, for every function $u : X\rightarrow\mathbb{R}$ that is integrable on balls, and for every upper gradient $g$ of $u$ in $X.$ For simplicity, we will from now on only consider the case of a $Q$-doubling space and we will assume that $p=Q.$\\
\indent Relying on \cite{KL02}, \cite{BO05}, and \cite{KZ08} one obtains the following conclusion.\\

\noindent\textbf{Theorem A.} Let $\epsilon>0.$ Let $(X,d,\mu)$ be a complete $Q$-doubling space with $Q>1$ that supports a $Q$-Poincar\'e inequality. If $u\in W^{1,Q}(X)$, then
\begin{equation}
u(x)=\lim_{r\to 0}\frac{1}{\mu(B(x,r))}\int_{B(x,r)}u(y)\, d\mu(y)
\end{equation}
outside a set $E_\epsilon$ with $H^{h}(E_\epsilon)=0,$ where $h(t)=\log^{1-Q-\epsilon}(1/t).$\\
\\
Theorem A is not explicitly stated in literature and thus let us describe how it follows from the indicated references. First of all, \cite{KL02} together with \cite{KZ08} gives the existence of Lebesgue points capacity
almost everywhere. Next, \cite{BO05} gives the desired relation between capacity and Hausdorff measure, but under the assumption that the space supports a $1$-Poincar\'e inequality. However, an examination of the corresponding proof in \cite{BO05} shows that it actually suffices that the Poincar\'e inequality \eqref{PI} holds for each $u\in W^{1,Q}(X)$ with $p=1$ for some function $g\in L^Q(X),$ whose $Q$-norm is at most a fixed constant times the infimum of $Q$-norms
of all upper gradients of $u.$ This requirement holds by the self-improving property of Poincar\'e inequalities \cite{KZ08}, see Section 4.\\
\indent The argument in the previous paragraph requires that $(X,d)$ be complete: the self-improving property from \cite{KZ08} may fail in the non-complete setting, see \cite{Kos99}. Moreover, even in the complete case, the self-improvement may fail unless we require a $Q$-Poincar\'e inequality for all $u\in W^{1,Q}(X).$ It is then natural to inquire if these two conditions
are necessary for the conclusion of Theorem A.\\
\indent Our result gives a rather optimal conclusion.\\

\noindent\textbf{Theorem B.} Let $\epsilon>0.$ Suppose that $(X,\mu)$ is a $Q$-doubling space for some $Q>1$. Assume that $X$ satisfies a chain condition (see definition \ref{chain}) and that the $p$-Poincar\'e inequality \eqref{PI} holds for a pair of functions $(u,g)$ with $p=Q$ where $g\in L^Q$ and $u$ is integrable on balls. Then
\begin{equation}\label{leb}
u(x)=\lim_{r\to 0}\frac{1}{\mu(B(x,r))}\int_{B(x,r)}u(y)\, d\mu(y)
\end{equation}
outside a set $E_\epsilon$ with $H^{h}(E_\epsilon)=0,$ where $h(t)=\log^{1-Q-\epsilon}(1/t).$\\
\indent As in the classical setting, the meaning of \eqref{leb} is that the limit exists outside $E_{\epsilon}$ and defines a representative for which \eqref{leb} holds outside $E_{\epsilon}.$\\

\indent Since the integral in \eqref{gauge} diverges for $h(t)=\log^{1-Q}(1/t),$ the conclusion
of Theorem B is rather optimal. We do not know if one could obtain the same conclusion as in the classical euclidean setting in this generality; under the assumptions of Theorem A one actually has a full analogue. Theorem B can be viewed as a refined version of a result in \cite{Giu69} on the existence of Lebesgue points that also avoids the use of capacities.\\

A doubling space that supports a $p$-Poincar\'e inequality is necessarily connected and even bi-Lipschitz equivalent to a geodesic space, if it is complete \cite{Che99}. Since each geodesic space satisfies a chain condition, the assumption of chain condition in Theorem B is natural. One can actually obtain the existence of a limit in \eqref{leb} outside a larger exceptional set even without a chain condition, see Section 3 below. This leads to gauge functions of the type $h(t)=\log^{-Q-\epsilon}(1/t).$\\
\indent This paper is organized as follows. We explain our notation and state a couple of preliminary results in Section 2. The proof of Theorem B is given in Section 3 and the proof of Theorem A in the appendix. 



\section{Notation and preliminaries}
We assume throughout that $X=(X,d,\mu)$ is a metric measure space equipped with a metric $d$ and a Borel regular outer measure $\mu.$ We call such a $\mu$ as a measure. The Borel-regularity of the measure $\mu$ means that all Borel sets are $\mu$-measurable and that for every set $A\subset X$ there is a Borel set $D$ such that $A\subset D$ and $\mu(A)=\mu(D).$\\

We denote open balls in $X$ with center $x\in X$ and radius $0<r<\infty$ by $$B(x,r)=\{y\in X : d(y,x)<r\}.$$
If $B=B(x,r)$ is a ball, with center and radius understood, and $\lambda>0,$ we write
$$\lambda B=B(x,\lambda r).$$
With small abuse of notation we write $\rad(B)$ for the radius of a ball $B$ and we always have $$\diam(B)\leq 2\rad(B),$$
and the inequality can be strict.\\

A Borel regular measure $\mu$ on a metric space $(X,d)$ is called a \textit{doubling measure} if every ball in $X$ has positive and finite measure and there exist a constant $C_{\mu}\geq 1$ such that
\begin{equation*}
\mu(B(x,2r))\leq C_{\mu}\,\mu(B(x,r))
\end{equation*}
for each $x\in X$ and $r>0.$ We call a triple $(X,d,\mu)$ a \textit{doubling metric measure space} if $\mu$ is a doubling measure on $X.$\\ 

If $A\subset X$ is a $\mu$-measurable set with finite and positive measure, then the \textit{mean value} of a function $u\in L^1(A)$ over $A$ is
$$u_A=\dashint_A u\, d\mu=\frac{1}{\mu(A)}\int_A u\, d\mu.$$
\\
\indent A metric space is said to be \textit{geodesic} if every pair of points in the space can be joined by a curve whose length is equal to the distance between the points.\\

\indent We recall that the \textit{generalized Hausdorff $h$-measure} is defined by
\begin{equation*}
H^h(E)=\limsup_{\delta\rightarrow 0}H_{\delta}^{h}(E),
\end{equation*}
where
\begin{equation*}
H_{\delta}^{h}(E)=\inf\left\{\sum h(\diam(B_i)) : E\subset\bigcup B_i,~\diam(B_i)\leq\delta\right\},
\end{equation*}
where the dimension gauge function $h$ is required to be continuous and increasing with $h(0)=0.$ In particular, if $h(t)=t^\alpha$ with some $\alpha>0,$ then $H^h$ is the usual \textit{$\alpha$-dimensional Hausdorff measure}, denoted also by $H^{\alpha}.$ See \cite{Rog98} for more information on the generalized Hausdorff measure.\\
\indent For the convenience of reader we state here a fundamental covering lemma (for a proof see \cite[2.8.4-6]{Fed69} or \cite[Theorem 1.3.1]{Zie89}).
\begin{lemma}[5B-covering lemma]\label{cover}
Every family $\mathcal{F}$ of balls of uniformly bounded diameter in a metric space $X$ contains a pairwise disjoint subfamily $\mathcal{G}$ such that for every $B\in\mathcal{F}$ there exists $B'\in\mathcal{G}$ with $B\cap B'\neq\emptyset$ and $\diam(B)<2\diam(B').$ In particular, we have that
$$\bigcup_{B\in\mathcal{F}}B\subset\bigcup_{B\in\mathcal{G}}5B.$$
\end{lemma}
The following lemma will be essential for the proof of Theorem B.
\begin{lemma}\label{3}
Suppose that $\{a_j\}_{j=0}^{\infty}$ is a sequence of non-negative real numbers such that $\sum_{j\geq 0}a_j<\infty$. Then
\begin{equation*}
\sum_{j\geq 0}\frac{a_j}{\left(\sum_{i\geq j}a_i\right)^{1-\delta}}<\infty\quad\text{for any}\quad 0<\delta<1.
\end{equation*}
\end{lemma}
\begin{proof}
For any $n\geq 1$, we use summation by parts (Newton series) and Bernoulli's inequality to obtain
\begin{eqnarray*}
\sum_{j=0}^n\frac{a_j}{\left(\sum\limits_{i\geq j}a_i\right)^{1-\delta}} &=& \frac{1}{\left(\sum\limits_{i\geq 0}a_i\right)^{1-\delta}}\sum_{j=0}^na_j + \sum_{j=0}^{n-1}\left(\frac{1}{\left(\sum\limits_{i\geq j+1}a_i\right)^{1-\delta}}-\frac{1}{\left(\sum\limits_{i\geq j}a_i\right)^{1-\delta}}\right)\sum_{k=j+1}^{n}a_k\\
&\leq & \frac{1}{\left(\sum\limits_{i\geq 0}a_i\right)^{1-\delta}}\sum_{j=0}^na_j + \sum_{j=0}^{n-1}\frac{(1-\delta)a_j}{\left(\sum\limits_{i\geq j}a_i\right)^{1-\delta}\left(\sum\limits_{i\geq j+1}a_i\right)}\sum_{k=j+1}^na_k.
\end{eqnarray*}
Now, if we let $n\rightarrow\infty$, we get
\begin{eqnarray*}
\sum_{j\geq 0}\frac{a_j}{\left(\sum\limits_{i\geq j}a_i\right)^{1-\delta}} \leq \left(\sum_{j\geq 0}a_j\right)^{\delta}+(1-\delta)\sum_{j\geq 0}\frac{a_j}{\left(\sum\limits_{i\geq j}a_i\right)^{1-\delta}}
\end{eqnarray*}
and hence
\begin{eqnarray*}
\sum_{j\geq 0}\frac{a_j}{\left(\sum\limits_{i\geq j}a_i\right)^{1-\delta}} \leq \frac{1}{\delta}\left(\sum_{j\geq 0}a_j\right)^{\delta}<\infty.
\end{eqnarray*}
\end{proof}
\section{Proof of Theorem B}
In this section, we give the proof of Theorem B. Let us begin with a weaker statement that does not require a chain condition. Thus assume only that $(X,\mu)$ is $Q$-doubling and that $(u,g)$ satisfies $Q$-Poincar\'e. Given $\epsilon>0,$ we wish to find $E_{\epsilon}\subset X$ with $H^h(E_{\epsilon})=0$ for $h(t)=\log^{-Q-\epsilon}(1/t)$ and so that the limit
\begin{equation*}
\lim_{r\rightarrow 0}\frac{1}{\mu(B(x,r))}\int_{B(x,r)}u(y)\, d\mu(y)
\end{equation*}
exists for $x$ outside $E_{\epsilon}.$\\
\indent Towards this end, it suffices to show that the sequence $\left(u_{B_j(x)}\right)_j$ of the integral averages of $u$ over the balls $B(x,2^{-j})$ is a Cauchy sequence outside such a set $E_{\epsilon}.$ Indeed, given $2^{-j-1}<r<2^{-j},$
\begin{eqnarray*}
\vert u_{B(x,r)}-u_{B(x,2^{-j})}\vert\leq \dashint_{B(x,r)}\vert u-u_{B_j(x)}\vert\leq C\dashint_{B_j(x)}\vert u-u_{B_j(x)}\vert\leq C\left(\int_{\lambda B_j(x)}g^Q\, d\mu\right)^{\frac{1}{Q}}
\end{eqnarray*}
by $Q$-doubling and $Q$-Poincar\'e. Similarly, for $l<m,$
\begin{equation*}
\vert u_{B_l(x)}-u_{B_m(x)}\vert\leq C\sum_{j=l}^{m-1}\left(\int_{\lambda B_j(x)}g^Q\right)^{\frac{1}{Q}}.
\end{equation*}
Hence, $\left(u_{B_j(x)}\right)_j$ is Cauchy provided $\int_{B(x,r)}g^Q\, d\mu\leq C\log^{-Q-\epsilon}(1/r)$ for all suffices small $r>0.$ By usual covering theorems, this holds outside a desired set.\\
\indent Towards the proof of Theorem B, we give a definition of a \textit{chain condition}, a version of which is already introduced in \cite{HK00}.
\begin{definition}\label{chain}
We say that a space $X$ satisfies a \textit{chain condition} if for every $\lambda\geq 1$ there are constants $M\geq 1,$ $0<m\leq 1$ such that for each $x\in X$ and all $0<r<\diam(X)/8$ there is a sequence of balls $B_0,B_1,B_2,\ldots$ with\\
1. $B_0\subset X\setminus B(x,r)$,\\
2. $M^{-1}\diam(B_i)\leq \dist(x,B_i)\leq M\diam(B_i)$,\\
3. $\dist(x,B_i)\leq Mr2^{-mi}$,\\
4. there is a ball $D_i\subset B_i\cap B_{i+1}$, such that $B_i\cup B_{i+1}\subset MD_i$,\\
for all $i\in\mathbb{N}\cup\{0\}$ and\\
5. no point of $X$ belongs to more than $M$ balls $\lambda B_i$.
\end{definition}
\indent The sequence $B_i$ will be called a \textit{chain associated with} $x,r$.\\

The existence of a doubling measure on $X$ does not guarantee a chain condition. In fact, such a space can be badly disconnected, whereas a space with a chain condition cannot have \lq\lq large gaps\rq\rq . For example, the standard $1/3$-Cantor set satisfies a chain condition only for $\lambda<2$. Here we show that a large number of spaces satisfy our chain condition.\\

Let $X$ be a metric space. For $0<r<R$ and $x\in X,$ we define the annulus $A(x,r,R)$ to be the set $\overline{B}(x,R)\setminus B(x,r).$ 
\begin{definition}\label{LAC}
A metric space $X$ is said to be $L$-annularly connected if whenever $y,z\in A(x,r,2r)$ for some $x\in X$ and $r>0,$ then there exists a curve joining $y$ and $z$ in $A(x,r/L,2rL).$ 
\end{definition}
Annular connectivity holds, for example, for complete doubling metric spaces that support a suitable Poincar\'e inequality \cite{HK00}, \cite{Kor07}.
\begin{lemma}
Suppose that $(X,d,\mu)$ is a doubling metric measure space, $(X,d)$ is connected and $L$-annularly connected. Then $(X,d)$ satisfies a chain condition.
\end{lemma}
\begin{proof}
Let $x\in X$ and $0<r<\diam(X)/8.$ Then $X\setminus\overline{B}(x,2r)\neq\emptyset.$ By connectivity, for each $j\geq 0$ there is $y_j\in X$ with $d(x,y_j)=2^{-j+1}r.$ Fix $0<\epsilon<1.$ As $\mu$ is doubling, we can cover each annulus $A_j(x)=A(x,2^{-j}r,2^{-j+1}r)$ by at most $N$ balls of radii equal to $\epsilon 2^{-j}r$ and the annulus $A(x,2r,2rL)$ by at most $N$ balls of radii equal to $\epsilon r$ with $N$ independent of $x,j.$ When $\epsilon$ is sufficiently small, depending only on $\lambda,$ the balls $2\lambda B$ with $B$ corresponding to $A_j(x)$ and $2\lambda B'$ with $B'$ corresponding to $A_i(x)$ do not intersect provided $\vert i-j\vert\geq 2.$ Since $(X,d)$ is annularly connected, we can connect the points $y_j,y_{j+1},$ $j\geq 0,$ by a curve in a wider annulus from definition \ref{LAC}. Collect all those balls from the collection above which intersect the curve joining $y_j$ and $y_{j+1},$ $j\geq 0.$ Consider the new collection of balls, where each ball chosen above gets replaced by the double of it, i.e. we replace $B(y,s)$ by $B(y,2s).$ Beginning with $y_0,$ we order our balls into a chain along the curves joining the points $y_j$ and $y_{j+1}.$ The desired properties follow, with $m=1/N$ for condition number $3.$ 
\end{proof}
Annular connectivity is not necessary for our chain condition. For example, the real line satisfies a chain condition, and so do geodesic spaces.
\begin{lemma}\label{geodesiclemma}
If $(X,d)$ is a geodesic space, then $(X,d)$ satisfies a chain condition.
\end{lemma}
Lemma \ref{geodesiclemma} follows from the proof of Lemma $8.1.6$ in \cite{HKST}.\\

By using the chain condition, the following lemma yields us the condition that we want for the proof of Theorem B.
\begin{lemma}\label{comparision}
Suppose that $X$ satisfies a \textit{chain condition} and let the sequence $B_i$ be a \textit{chain associated with} $x,R_2$ for $x\in X$ and $0<R_2<\diam(X)/8$. Let $0<R_1<R_2.$ Then we can find balls $B_{i_{R_2}},B_{i_{R_2}+1},\ldots,B_{i_{R_1}}$ from the above collection such that
\begin{eqnarray}
\frac{R_2}{M(1+M)^2} & \leq \diam(B_{i_{R_2}}) & \leq MR_2,\\
\frac{R_1}{M(1+M)^2} & \leq \diam(B_{i_{R_1}}) & \leq MR_1
\end{eqnarray}
hold and $B_{i_{R_2}} \subset B(x,R_2),$ $B_{i_{R_1}}\subset B(x,R_1)$ and also the balls $B_{i_{R_2}},B_{i_{R_2}+1},\ldots,B_{i_{R_1}}$ form a chain.
\end{lemma}
\begin{proof}
Let $i_{R_2}=\min\{i\geq 0: B_i\subset B(x,R_2)\}.$ Hence we have $\dist(x,B_{i_{R_2}})\leq R_2,$ which implies that $\diam(B_{i_{R_2}})\leq MR_2,$  using the second condition of the above definition. Again $B_{i_{R_2}-1}\cap\left(X\setminus B(x,R_2)\right)\neq\emptyset.$ Using the triangle inequality, we obtain $\dist(x,B_{i_{R_2}-1})+\diam(B_{i_{R_2}-1})\geq R_2$ and hence we have $$\diam(B_{i_{R_2}-1})\geq\frac{R_2}{1+M}.$$
Since $B_{i_{R_2}}\cap B_{i_{R_2}-1}\neq\emptyset,$ we write $\dist(x,B_{i_{R_2}})+\diam(B_{i_{R_2}})\geq\dist(x,B_{i_{R_2}-1})$ and hence we obtain
$$\diam(B_{i_{R_2}})\geq\frac{R_2}{M(1+M)^2}.$$
Once $B_{i_{R_2}}$ is chosen, we can choose $B_{i_{R_2}+1},B_{i_{R_2}+2},\ldots,B_{i_{R_1}}$ from the above collection, where $i_{R_1}=\min\{i\geq 0: B_i\subset B(x,R_1)\}.$ Then obtain the above estimates for $B_{i_{R_1}}$ in a similar way and clearly the new collection of balls form a chain.
\end{proof}
Our next lemma shows that we have an upper bound for the volume quotient in \eqref{1} under the chain condition.
\begin{lemma}\label{measure}
Suppose that a doubling metric measure space $(X,d,\mu)$ satisfies a chain condition. Then there is an exponent $\tilde{Q}>0$ and a constant $C_0\geq 1$ so that
\begin{equation}\label{UP}
\frac{\mu(B(x,s))}{\mu(B(a,r))}\leq C_0\left(\frac{s}{r}\right)^{\tilde{Q}}
\end{equation}
holds whenever $a\in X$, $x\in B(a,r)$ and $0<s\leq r.$
\end{lemma}
\begin{proof}
Let $B$ be an arbitrary ball in $X.$ We choose $\tau<1/2$ such that we get a ball $\tilde{B}\subset B$ disjoint from $\tau B$ using the chain condition and hence using the doubling property we obtain
\begin{eqnarray*}
\mu(B) &\geq & \mu(\tau B)+\mu(\tilde{B})\\
&\geq & \mu(\tau B)+C_{\mu}\,\mu(B),
\end{eqnarray*}
which means that we have the \lq\lq reverse\rq\rq\ doubling condition
\begin{equation*}
\mu(\tau B)\leq (1-C_{\mu})\mu(B).
\end{equation*}
Then a simple iteration argument gives us the required condition. 
\end{proof}
It immediately follows from Lemma \ref{measure} that $H^h(E_\epsilon)=0$ implies, in the setting of Theorem B, that $\mu(E_\epsilon)=0.$ Hence the conclusion of Theorem B has content.\\

\begin{proof}[\textbf{Proof of Theorem B}] Let $x\in X$. For given $0<r<1$, we can always find $j\in\mathbb{N}$ such that $2^{-(j+1)}<r<2^{-j}$. It is enough to consider the balls $B(x,2^{-j})$ instead of $B(x,r)$ as we have, using the doubling property and the Poincar\'e inequality,
\begin{eqnarray*}
\vert u_{B(x,r)}-u_{B(x,2^{-j})}\vert & \leq & \dashint_{B(x,r)}\vert u-u_{B(x,2^{-j})}\vert\, d\mu\\
& \leq & c\dashint_{B(x,2^{-j})}\vert u-u_{B(x,2^{-j})}\vert\, d\mu\\
& \leq & c\left(\int_{B(x,2^{-j})}g^Q\, d\mu\right)^{\frac{1}{Q}}\rightarrow 0 ~\text{as}~ j\rightarrow\infty.
\end{eqnarray*}
Our aim is to show that the sequence $u_{B(x,2^{-j})}=\dashint_{B(x,2^{-j})}u(y)\, d\mu(y)$, $j\in\mathbb{N}$ is a Cauchy sequence. Towards this end, for $m,l\in\mathbb{N},m>l$, let us consider the difference
\begin{equation*}
\vert u_{B(x,2^{-m})}-u_{B(x,2^{-l})}\vert\leq\vert u_{B(x,2^{-l})}-u_{B_{i_l}}\vert +\vert u_{B(x,2^{-m})}-u_{B_{i_m}}\vert +\vert u_{B_{i_l}}-u_{B_{i_m}}\vert ,
\end{equation*}
where the balls $B_{i_l},B_{i_l+1},\ldots,B_{i_m}$ are obtained from Lemma \ref{comparision} for $R_1=2^{-m},R_2=2^{-l}$.
Using the doubling property, Poincar\'e inequality and Lemma \ref{comparision}, we obtain
\begin{eqnarray*}
\vert u_{B_{i_l}}-u_{B(x,2^{-l})}\vert & \leq & \dashint_{B_{i_l}}\vert u-u_{B(x,2^{-l})}\vert\, d\mu\\
& \leq & c\dashint_{B(x,2^{-l})}\vert u-u_{B(x,2^{-l})}\vert\, d\mu\\
& \leq & c\left(\int_{B(x,2^{-l})}g^Q\, d\mu\right)^{\frac{1}{Q}}\rightarrow 0 ~\text{as}~ l\rightarrow\infty
\end{eqnarray*}
and similarly we get $\vert u_{B(x,2^{-m})}-u_{B_{i_m}}\vert\rightarrow 0 ~\text{as}~ m\rightarrow\infty.$ So, it is enough to prove that $\vert u_{B_{i_l}}-u_{B_{i_m}}\vert\rightarrow 0$ when both $m,l$ tend to infinity.\\
\indent Fix $\epsilon>0$ and write $h_1(t)=\log^{1-Q-\frac{\epsilon}{2}}(1/t).$ Let $\tilde{\epsilon}>0,$ which is to be chosen later.  We use a telescopic argument for the balls $B_{i_l},B_{i_l+1},\ldots,B_{i_m}$ and also use chain conditions, relative lower volume decay \eqref{1} and Poincar\'e inequality \eqref{PI} to estimate
\begin{eqnarray*}
\vert u_{B_{i_m}}-u_{B_{i_l}}\vert & \leq & \sum_{n=i_l}^{i_m-1}\vert u_{B_n}-u_{B_{n+1}}\vert\\
&\leq & \sum_{n=i_l}^{i_m-1}\left(\vert u_{B_n}-u_{D_n}\vert+\vert u_{B_{n+1}}-u_{D_n}\vert\right)\\
&\leq & \sum_{n=i_l}^{i_m-1}\left(\dashint_{D_n}\vert u-u_{B_n}\vert\, d\mu +\dashint_{D_n}\vert u-u_{B_{n+1}}\vert\, d\mu\right)\\
& \leq & c\sum_{n=i_l}^{i_m-1}\dashint_{B_n}\vert u-u_{B_n}\vert\, d\mu\\
& \leq & c\sum_{n=i_l}^{i_m-1}\diam(B_n)\left(\dashint_{\lambda B_n}g^Q\, d\mu\right)^{\frac{1}{Q}}\\
& \leq & c\sum_{n\geq i_l}\left(\frac{\diam(B_n)^Q}{\mu(B_n)}\int_{\lambda B_n}g^Q\, d\mu\right)^{\frac{1}{Q}}n^{\frac{Q-1+\tilde{\epsilon}}{Q}}n^{-\frac{Q-1+\tilde{\epsilon}}{Q}}\\
& \leq & c\left(\sum_{n\geq i_l}\frac{\diam(B_n)^Q}{\mu(B_n)}n^{Q-1+\tilde{\epsilon}}\int_{\lambda B_n}g^Q\, d\mu\right)^{\frac{1}{Q}} \left(\sum_{n\geq i_l}n^{-\frac{Q-1+\tilde{\epsilon}}{Q-1}}\right)^{\frac{Q-1}{Q}}\\
& \leq & \frac{ci_l^{-\frac{\tilde{\epsilon}}{Q}}}{\mu(B(x,1))}\left(\sum_{n\geq i_l}n^{Q-1+\tilde{\epsilon}}\int_{\lambda B_n}g^Q\, d\mu\right)^{\frac{1}{Q}}.
\end{eqnarray*}
Now we consider the convergence of the sum $\sum_{n\geq i_l}n^{Q-1+\tilde{\epsilon}}\int_{\lambda B_n}g^Q\, d\mu$ when $l$ is large. If we have
\begin{equation}\label{2}
\int_{B(x,r)}g^Q\, d\mu\leq \frac{c}{\log^{Q-1+\epsilon/2}\left(\frac{1}{r}\right)}
\end{equation}
for all sufficiently small $0<r<1/5,$ then
$$\sum_{n'\geq n}\int_{\lambda B_{n'}}g^Q\, d\mu\leq \frac{c}{n^{Q-1+\epsilon/2}}$$
for all $n\geq i_l,$ provided $l$ is sufficiently large. Then we choose $\tilde{\epsilon}=\frac{\epsilon}{2}-\delta(Q-1+\frac{\epsilon}{2})$ for some $0<\delta<1$ (we can choose $\delta$ as small as we want to make $\tilde{\epsilon}$ positive) and use Lemma \ref{3} to obtain
$$\sum_{n\geq i_l}n^{Q-1+\tilde{\epsilon}}\int_{\lambda B_n}g^Q\, d\mu<\infty.$$
Hence we get $\vert u_{B_{i_m}}-u_{B_{i_l}}\vert\rightarrow 0$ when both $l,m$ tend to infinity.\\
\indent On the other hand, let us consider the set
\begin{multline*}
E_\epsilon=\bigg\{x\in X :\text{there exists arbitrarily small}~0<r_x<\frac{1}{5}~\text{such that}\\ \int_{B(x,r_x)}g^Q\, d\mu\geq\frac{c}{\log^{Q-1+\epsilon/2}\left(\frac{1}{r_x}\right)}\bigg\}.
\end{multline*}
Let $0<\delta_1<1/5.$ Then we get a pairwise disjoint family $\mathcal{G},$ by the using 5B-covering lemma, such that
$$E_{\epsilon}\subset\bigcup_{B\in\mathcal{G}}5B,$$
where $\diam(B)<2\delta_1$ for $B\in\mathcal{G}.$ Then we estimate
\begin{eqnarray*}
\mathcal{H}_{10\delta_1}^{h_1}(E_{\epsilon}) & \leq & \sum_{B\in\mathcal{G}} \log^{1-Q-\frac{\epsilon}{2}}\left(\frac{1}{5 \rad(B)}\right)\\
& \leq & c\sum_{B\in\mathcal{G}} \log^{1-Q-\frac{\epsilon}{2}}\left(\frac{1}{\rad(B)}\right)\\
& \leq & c\sum_{B\in\mathcal{G}} \int_{B}g^Q\, d\mu\\
& \leq & c\int_{\bigcup\limits_{B\in\mathcal{G}}B}g^Q\, d\mu<\infty.
\end{eqnarray*}
It follows that $\mathcal{H}^{h_1}(E_{\epsilon})<\infty$ and hence we have that $\mathcal{H}^h(E_{\epsilon})=0$ (see \cite[Theorem 40]{Rog98}), which gives us the existence of $ \lim\limits_{i\rightarrow\infty}\dashint_{B_i}u(y)\, d\mu(y)$ for $\mathcal{H}^h$-a.e. $x\in X.$ Since $u$ is locally integrable, $\mu$-almost every $x$ is a Lebesgue point, and hence \eqref{leb} extends to hold $H^h$-a.e. for a representative of $u.$ 
\end{proof}
\begin{remark}
The proof of Theorem B actually only requires a chain condition for the value of $\lambda$ given in our assumption \eqref{PI} on the pair $(u,g).$
\end{remark}
\section{Appendix}
In this section, we complete the proof of Theorem A. First we recall the definition of maximal functions and a version of the well-known maximal theorem of Hardy, Littlewood and Wiener.\\
\indent Let $(X,d,\mu)$ be a metric measure space. The \textit{Hardy-Littlewood maximal function} $Mf$ of a locally integrable function $f$ is the function defined by
\begin{equation}
Mf(x):=\sup_{r>0}\dashint_{B(x,r)}\vert f(y)\vert\, d\mu(y)
\end{equation}
and the \textit{restricted maximal function} is defined by
\begin{equation}
M_Rf(x):=\sup_{0<r<R}\dashint_{B(x,r)}\vert f(y)\vert\, d\mu(y)
\end{equation}
for $R>0$ fixed.\\
Here we only state the Maximal theorem, for a proof see \cite{Smi56}, \cite{Rau56} or \cite{Hei01}.
\begin{theorem}[Maximal theorem]\label{Mt}
Let X be a doubling metric measure space. There exist constants $C_p,$ depending only on $p$ and on the doubling constant of $\mu,$ such that
\begin{equation}
\mu(\{x:Mf(x)>t\})\leq \frac{C_1}{t}\Vert f\Vert_{L^1(X)}
\end{equation}
for all $t>0$ and that
\begin{equation}
\Vert Mf\Vert_{L^p(X)}\leq C_p\Vert f\Vert_{L^p(X)}
\end{equation}
for all $1<p\leq\infty$ and for all measurable functions $f.$
\end{theorem}
We also recall here the Haj\l asz-Sobolev space $M^{1,p}(X)$ defined by Haj\l asz in \cite{Haj96}. A measurable function $u : X\rightarrow\mathbb{R}$ belongs to the Haj\l asz-Sobolev space $M^{1,p}(X)$ if and only if $u\in L^p(X)$ and there exists a nonnegative function $g\in L^p(X)$ such that the inequality
\begin{equation*}
\vert u(x)-u(y)\vert\leq d(x,y)(g(x)+g(y))
\end{equation*}
holds for all $x,y\in X\setminus E,$ where $\mu(E)=0.$\\
The following theorem completes the sketch of the proof of Theorem A from our introduction.
\begin{theorem}
Suppose that $(X,d,\mu)$ is a complete and doubling space that supports a $Q$-Poincar\'e inequality. Let $u\in W^{1,Q}(X)$ and $g$ be its upper gradient. Then there exists a function $h\in L^Q(X)$ such that the inequality
\begin{equation}\label{SPI}
\dashint_B\vert u-u_B\vert\leq Cr\dashint_Bh\, d\mu
\end{equation}
holds on every ball $B$ of radius $r$ and that $\Vert h\Vert_Q\leq c\Vert g\Vert_Q.$
\end{theorem}
\begin{proof}
By \cite{KZ08}, we know that there exists $\epsilon>0$ such that a $(Q-\epsilon)$-Poincar\'e inequality holds for the pair $(u,g).$ Then Theorem 3.1 of \cite{HK00} shows that $u\in M^{1,Q}(X)$ and in particular, Theorem 3.2 of \cite{HK00} gives us the pointwise inequality
\begin{equation}\label{PoI}
\vert u(x)-u(y)\vert\leq Cd(x,y)\left(h(x)+h(y)\right)
\end{equation}
for almost every $x,y\in X,$ where $h(x)=(M_{2\lambda d(x,y)}g^{Q-\epsilon}(x))^{\frac{1}{Q-\epsilon}}.$ Now integrating inequality \eqref{PoI} over a ball $B$ with respect to $x$ and $y$, we obtain inequality \eqref{SPI} and using Maximal Theorem \ref{Mt} we obtain
\begin{equation*}
\Vert h\Vert_Q=\Vert Mg^{Q-\epsilon}\Vert_{\frac{Q}{Q-\epsilon}}^{\frac{1}{Q-\epsilon}}\leq c\Vert g^{Q-\epsilon}\Vert_{\frac{Q}{Q-\epsilon}}^{\frac{1}{Q-\epsilon}}=c\Vert g\Vert_Q.
\end{equation*}
Note that we have used the fact that $g^{Q-\epsilon}\in L^{\frac{Q}{Q-\epsilon}},$ $\frac{Q}{Q-\epsilon}>1$ and that
\begin{equation*}
\left(M_{2\lambda d(x,y)}g^{Q-\epsilon}(x)\right)^{\frac{1}{Q-\epsilon}}\leq \left(Mg^{Q-\epsilon}(x)\right)^{\frac{1}{Q-\epsilon}}.
\end{equation*}
\end{proof}
\def\bibname{References}
\bibliography{lebesguepoint}
\bibliographystyle{alpha}
\end{document}